\documentclass[reqno]{amsart}
\usepackage{times}
\usepackage{amssymb}
\usepackage{amsmath}
\allowdisplaybreaks[4]

\newtheorem{theorem}{Theorem}
\newtheorem{lemma}[theorem]{Lemma}
\newtheorem{proposition}[theorem]{Proposition}
\newtheorem{corollary}[theorem]{Corollary}
\newtheorem{example}[theorem]{Example}

\theoremstyle{remark}

\newcommand{\C}{{\mathbb C}}
\newcommand{\cn}{\C^n}
\newcommand{\N}{{\mathbb N}}

\newcommand{\re}{{\rm Re}\,}

\begin{document}

\title[Monomial-type Toeplitz operators on some weakly pseudoconvex domains]
{Commuting and semi-commuting monomial-type Toeplitz operators on some weakly pseudoconvex domains}

\author[C. Jiang, X.T. Dong and Z.H. Zhou]{Cao Jiang, Xing-Tang Dong and Ze-Hua Zhou$^*$}

\address{School of Mathematics, Tianjin University, Tianjin 300354, P.R. China.}
\email{jiangcc96@163.com}

\address{School of Mathematics, Tianjin University, Tianjin 300354, P.R. China.}
\email{dongxingtang@163.com}

\address{School of Mathematics, Tianjin University, Tianjin 300354, P.R. China.}
\email{zehuazhoumath@aliyun.com; zhzhou@tju.edu.cn}


\date{}

\thanks{\noindent $^*$Corresponding author.\\
This work was supported in part by the National Natural Science Foundation of
China (Grant Nos. 11201331; 11371276; 11771323).}

\begin{abstract}
In this paper, we completely characterize the finite rank commutator and semi-commutator of two
monomial-type Toeplitz operators on the Bergman space of certain weakly pseudoconvex domains.
Somewhat surprisingly, there are not only plenty of commuting monomial-type Toeplitz operators but also non-trivial semi-commuting monomial-type Toeplitz operators. Our results are new even for the unit ball.
\end{abstract} \maketitle

\section{Introduction}
The Toeplitz operators on certain pseudoconvex domains in $\cn$ have been the object of much study. See \cite{ARS,CDR,JK,QS} for example.
In this paper, we shall consider Toeplitz operators on the weakly pseudoconvex domains
$$
\Omega_m^n=\left\{(z_1,\cdots,z_n)\in\mathbb{C}^n:\sum_{i=1}^n{|z_i|^{2m_i}}<1\right\},
$$
where $m=(m_1,\cdots,m_n)$ is an $n$-tuple of positive integers. We shall suppose $n>1$ to avoid trivialities throughout the paper.
Then for each $z=(z_1,\cdots,z_n)\in\cn$, we denote
$$r=\sqrt{|z_1|^{2m_1}+\cdots+|z_n|^{2m_n}}$$ and
$$\zeta=(\zeta_1,\cdots,\zeta_n)
=\left(\frac{z_1}{r^{\frac{1}{m_1}}},\cdots,\frac{z_n}{r^{\frac{1}{m_n}}}\right)\in\mathbb{S}_m^n,$$ where
$\mathbb{S}_m^n$ is the boundary of $\Omega_m^n$. Note that these expressions define a set of
coordinates $(r,\zeta)$ for every $z\in\mathbb{C}$, these coordinates are called $m$-polar coordinates (see \cite{QS}). 
If $m=(1,\cdots,1)$, then $\Omega_m^n=\mathbb{B}^n$ is the unit ball centered at the origin.

Let $L^2(\Omega_m^n)$ be the collection of all square integrable functions with respect to the usual Lebesgue measure
$dV$ on $\Omega_m^n$. The Bergman space $A^2(\Omega_m^n)$ is the closed subspace of $L^2(\Omega_m^n)$ consisting of
holomorphic functions in $\Omega_m^n$.
Denote by $P: L^2(\Omega_m^n)\rightarrow A^2(\Omega_m^n)$ the orthogonal projection.
Given a symbol $u\in L^\infty(\Omega_m^n)$, the Toeplitz operator $T_u$ induced by $u$ is the bounded operator defined by
$$T_u(f)=P(uf):A^2(\Omega_m^n)\rightarrow A^2(\Omega_m^n).$$
For two Toeplitz operators $T_{f_1}$ and $T_{f_2}$ on $A^2(\Omega_m^n)$, the commutator
and semi-commutator are defined by
$[T_{f_1},T_{f_2}]=T_{f_1}T_{f_2}-T_{f_2}T_{f_1}$
and $(T_{f_1},T_{f_2}]=T_{f_1}T_{f_2}-T_{f_1f_2}$, respectively.

The problem of characterizing
when two Toeplitz operators commute or semi-commute on the Bergman space over various domains has been a long-term research topic.
For example, on the setting of the Bergman space over the unit disk,
the Brown-Halmos type theorems were obtained in \cite{AhC} and \cite{AxC} for Toeplitz operators with harmonic symbols, and many other types of (semi-)commuting Toepltz operators with quasihomogeneous symbols were found in \cite{CR, DZ6, LSZ, LZ1}.
However, the general (semi-)commuting problem remains open on the unit disk, and it becomes even more delicate and more
challenging on higher-dimensional balls, see \cite{DZ1, Le, Le1, V, Z1, ZD} for example.

Just recently, the second author and Zhu in \cite{DZhu} completely characterized when the commutators
and semi-commutators of two monomial Toeplitz operators on the Bergman space of the unit
ball $A^2(\mathbb{B}^n)$ have finite rank. In this paper,
we take the weakly pseudoconvex domain $\Omega_m^n$ as our domain and consider more general symbols, namely, the monomial-type symbols.
Recall that the monomial-type symbol is the function $\varphi:\Omega_m^n\rightarrow \mathbb{C}$ given by
$$\varphi(z)=r^l\zeta^p\overline{\zeta}^q$$
for $p,q\in\mathbb{N}^n$, $l\in \mathbb{R}_+$ (Here $\mathbb{R}_+$ denotes the set of all nonnegative real numbers). In this case, the corresponding
Toeplitz operator $T_{\varphi}$ is called a monomial-type Toeplitz operator.

To state our main results, we need some notations.
For $\alpha=(\alpha_1,\cdots,\alpha_n)\in\mathbb{N}^n$,
we write
$$|\widehat{\alpha}|=\frac{\alpha_1}{m_1}+\cdots+\frac{\alpha_n}{m_n}.$$ If $m_i=1$ for all $i\in\{1,\cdots,n\}$, then we will use the usual notation $|\alpha|$ to take instead of $|\widehat{\alpha}|$.
A tuple $(x_1,x_2,y_1,y_2)\in\mathbb{R}_+^4$ is called to satisfy Condition (I) if at least one of the following conditions holds.
\begin{enumerate}
  \item[(i)] $x_1=x_2=0$,
  \item[(ii)] $y_1=y_2=0$,
  \item[(iii)] $x_1=y_1=0$,
  \item[(iv)] $x_2=y_2=0$,
  \item[(v)] $x_1=x_2$ and $y_1=y_2$,
  \item[(vi)] $x_1=y_1$ and $x_2=y_2$.
\end{enumerate}
For $p, q, s, t\in\mathbb{N}^n$, Theorem A of \cite{DZhu} shows that the operators $T_{z^p\overline z^q}$ and $T_{z^s\overline z^t}$ commute on $A^2(\mathbb{B}^n)$ if and only if one of the following five conditions holds
\begin{enumerate}
\item[(c1)] One of the two operators is the identity operator.
\item[(c2)] Both operators have analytic symbols.
\item[(c3)] Both operators have conjugate analytic symbols.
\item[(c4)] $|p|=|q|$, $|s|=|t|$, and $(p_i, q_i, s_i, t_i)$ satisfies Condition (I) for all $i\in\{1,2,\cdots,n\}$.
\item[(c5)] $|p|=|s|$, $|q|=|t|$, and $(p_i, q_i, s_i, t_i)$ satisfies Condition (I) for all $i\in\{1,2,\cdots,n\}$.
\end{enumerate}
Clearly, conditions (c4) and (c5) produce lots of non-trivial commuting monomial Toeplitz
operators on $A^2(\mathbb{B}^n)$, see \cite[Example 6]{DZhu}.
On the setting of the Bergman space over the weakly pseudoconvex domain $\Omega_m^n$,
Barranco and Nungaray \cite{QS} studied the commutativity of two Toeplitz operators with special k-quasi-homogeneous symbols. Here $k=(k_1,\cdots,k_\tau)$ is a partition of $n$, and then each $z\in\cn$ can be decomposed into $\tau$ pieces.
Inspired by \cite[Theorem A]{DZhu}, 
we obtain the following general result which gives two non-trivial sufficient conditions for two k-quasi-homogeneous Toeplitz operators commuting on $A^2(\Omega_m^n)$.
\begin{proposition} Let $p, q, s, t\in\mathbb{N}^n$, and let $\varphi$ and $\psi$ be bounded k-quasi-radial functions on $\Omega_m^n$. Suppose one of the following statements holds:
\begin{enumerate}
  \item[(i)] $(p_i, q_i, s_i, t_i)$ satisfies Condition (I) for each $i\in\{1,2,\cdots,n\}$, $|p_{(j)}|=|q_{(j)}|$ and $|s_{(j)}|=|t_{(j)}|$ for each $j\in\{1,2,\cdots,\tau\}$.
  \item[(ii)] $(p_i, q_i, s_i, t_i)$ satisfies Condition (I) for each $i\in\{1,2,\cdots,n\}$, $|p_{(j)}|=|s_{(j)}|$ and $|q_{(j)}|=|t_{(j)}|$ for each $j\in\{1,2,\cdots,\tau\}$, and $\varphi(r_1,\cdots,r_\tau)=\psi(r_1,\cdots,r_\tau)$.
\end{enumerate}
Then the operators $T_{\zeta^{p}\overline{\zeta}^{q}\varphi}$ and $T_{\zeta^{s}\overline{\zeta}^{t}\psi}$ commute on $A^2(\Omega_m^n)$.
\end{proposition}
Since the proof of this result is a direct calculation using \cite[Lemma 3.6]{QS}, we leave it to the interested reader.
Unfortunately, so far no other non-trivial case was obtained for two k-quasi-homogeneous Toeplitz
operators commuting even on $A^2(\mathbb{B}^n)$.
In this paper, we will give many new commuting Toeplitz
operators on $A^2(\Omega_m^n)$ induced by the monomial-type symbols.


Our first main result not only shows that there is no nonzero finite rank commutator of two monomial-type
Toeplitz operators on $A^2(\Omega_m^n)$,
but also gives the sufficient and necessary condition for such two operators to be commutative.
\begin{theorem}\label{thm1}
Let $l,k\in\mathbb{R}_+$, $p,q,s,t\in\mathbb{N}^n$. Then the following statements are equivalent.
\begin{enumerate}
\item[(a)] The commutator $[T_{r^l\zeta^p\overline{\zeta}^q},T_{r^k\zeta^s\overline{\zeta}^t}]$ on
 $A^2(\Omega_m^n)$ has finite rank.
\item[(b)] The operators $T_{r^l\zeta^p\overline{\zeta}^q}$ and $T_{r^k\zeta^s\overline{\zeta}^t}$ commute on $A^2(\Omega_m^n)$.
\item[(c)] $(p_i,q_i,s_i,t_i)$ satisfies condition (I) for all $1\leq i\leq n$,
and there exist some real numbers $\mu, \nu$ and $a\geq \mu/2, b\geq \nu/2$ such that
 \begin{align}\label{eq1}
&\frac{\Gamma\left(\eta+|\widehat{p}|\right)
\Gamma\left(\eta+\nu+1\right)
\Gamma\left(\eta+\mu+|\widehat{s}|\right)}
{\Gamma\left(\eta+|\widehat{s}|\right)\Gamma\left(\eta+\mu+1\right)
\Gamma\left(\eta+\nu+|\widehat{p}|\right)}=\frac
{\left(\eta+b\right)\left(\eta+a+\nu\right)}{\left(\eta+a\right)\left(\eta+b+\mu\right)}
 \end{align}
 for any $\eta\in\mathbb{C}$ on some right half-plane. In this case,
 $|\widehat{q}|=|\widehat{p}|-\mu, |\widehat{t}|=|\widehat{s}|-\nu, l=2a-\mu$ and $k=2b-\nu$.
\end{enumerate}
\end{theorem}

While one would wish for a more conceptual conclusion,
there exist too many cases for the tuple $\left(|\widehat{p}|, |\widehat{s}|, \mu, \nu, a, b\right)$ satisfying the identity \eqref{eq1}.
Consequently, the commuting monomial-type Toeplitz operators on $A^2(\Omega_m^n)$ involves not
only the predictable cases but also plenty of non-trivial cases (see Corollary~\ref{cor five} and Example~\ref{ex nontri}).
It is also worth to mention that the weakly pseudoconvex domains $\Omega^n_m$ have many rich structural characteristics. For example, two of these domains are biholomorphically equivalent only if the corresponding $m_i$'s are the same up to permutation (see \cite{W} for example).
However, the operator theory on such family of domains becomes even more complicated and interesting (see Example~\ref{ex nontri2}).

As a consequence of Theorem~\ref{thm1}, some interesting necessary conditions for two monomial-type Toeplitz operators commuting on $A^2(\Omega_m^n)$ are obtained (see Corollary~\ref{cor nec}). Consequently, there exist a finite number of the combination of the tuple $(|\widehat{p}|,|\widehat{s}|,\mu,\nu,a,b)$ satisfying the identity \eqref{eq1}. Also, we get some differences of the operator theory between on the weakly pseudoconvex domains and on the unit ball.


For the finite rank semi-commutator problem, \cite[Theorem B]{DZhu} showed that the semi-commutator of two monomial Toeplitz operators on $A^2(\mathbb{B}^n)$ has finite
rank only in trivial cases: the first operator has conjugate analytic symbol or the second
operator has analytic symbol. By contrast, our second result produces non-trivial cases for semi-commuting monomial-type Toeplitz operators on $A^2(\Omega_m^n)$.
\begin{theorem}\label{thm2}
Let $l,k\in\mathbb{R}_+$, $p,q,s,t\in\mathbb{N}^n$. Then the following statements are equivalent:
\begin{enumerate}
\item[(a)] The semi-commutator $\left(T_{r^l\zeta^p\overline{\zeta}^q}, T_{r^k\zeta^s\overline{\zeta}^t}\right]$
      on $A^2(\Omega_m^n)$ has finite rank.
\item[(b)] $T_{r^l\zeta^p\overline{\zeta}^q}T_{r^k\zeta^s\overline{\zeta}^t}= T_{r^{l+k}\zeta^{p+s}\overline{\zeta}^{q+t}}$ on $A^2(\Omega_m^n)$.
\item[(c)] At least one of the following conditions holds:
\begin{enumerate}
  \item[(i)] $l=|\widehat{q}|-|\widehat{p}|$ and $t=0$;
  \item[(ii)] $k=|\widehat{s}|$ and $t=0$;
  \item[(iii)] $l=|\widehat{q}|$ and $p=0$;
  \item[(iv)] $k=|\widehat{s}|-|\widehat{t}|$ and $p=0$.
\end{enumerate}
\end{enumerate}
\end{theorem}

From Theorem~\ref{thm2} we get that the semi-commutator of two monomial Toeplitz operators $T_{z^p\overline{z}^q}$ and $T_{z^s\overline{z}^t}$ on $A^2(\Omega_m^n)$ has finite rank
if and only if either $p=0$ or $t=0$, which is parallel to the result in \cite[Theorem B]{DZhu}.
In sharp contrast to the monomial case, there exist many non-trivial semi-commuting monomial-type Toeplitz operators. Furthermore, it is interesting to observe that
the operator $T_{r^{\left(|\widehat{q}|-|\widehat{p}|\right)}\zeta^p\overline{\zeta}^q}$ semi-commutes with all Toeplitz operators induced by symbol functions of the form $r^k\zeta^s$.

In addition, we would like to mention that the situation in the unit disk is quite different from the case of higher dimensional weakly pseudoconvex domains.
In fact, the commutators or semi-commutators of two monomial-type Toeplitz operators on the Bergman space of the unit disk have finite
rank only in several trivial cases, see \cite[Corollaries 11 and 12]{DZ6}.

While the main idea of our method of proofs is adapted from \cite{DZhu}, substantial
amount of unexpected analysis is required to overcome some different nature of the
weakly pseudoconvex domains and monomial-type symbols. For example, the method of characterizing (semi-)commuting k-quasi-homogeneous Toeplitz
operators relies on explicit formulas for the action of the operators on the orthonormal
basis of the Bergman space. This action leads to a holomorphic identity on a domain in the complex n-space,
see \eqref{eq8} for example. 
Then the most critical step in the proof is the analysis of such identity.
To deal with this, \cite{DZhu} depended on the distribution of zeros of a holomorphic function,
which does not work in the case of monomial-type symbols. In this paper, we devise a
completely new approach, which simplify the proof greatly even though the domain and the operators seem more general and complicated
than those in \cite{DZhu}. Roughly speaking, we rewrite the identity \eqref{eq8} as \eqref{eq9}
and show that each function appeared in \eqref{eq9} must be a constant multiple of the exponent function.
Since all of them are bounded on the right half-plane with a certain growth at infinity, they must be the constant function.

We end this introduction by mentioning that
the implication (a) to (b) of Theorem~\ref{thm1} or Theorem~\ref{thm2} is in fact predictable: on the Bergman space, finite rank
commutators and semi-commutators with bounded symbols are usually zero. For example, in the case of the Bergman space of the
unit disk, Guo, Sun and Zheng \cite{GSZ} showed that there is no non-zero finite rank commutator or semi-commutator
of Toeplitz operators induced by bounded harmonic symbols.
But the problem is very much open for other symbol classes.




\section{Proofs of Theorems~\ref{thm1} and \ref{thm2}}

We start this section with the following proposition which will play an essential role in the proofs of our main theorems.

\begin{proposition}\label{uui}
Suppose each function $f_i$, $i\in\{1,\cdots, n\}$, is analytic and has no zero on certain right half-plane $\prod_{a_i}^+=\{\zeta_i\in\mathbb{C}: \re \zeta_i>a_i\}$ with $a_i\in\mathbb{R}$. If there exists an analytic function $f$ such that
  \begin{equation}\label{eq2}
    f(\zeta_1+\cdots+\zeta_n)=\prod\limits_{i=1}^n f_i(\zeta_i)
  \end{equation}
for any $\zeta_i\in\prod_{a_i}^+$, then
$$f_i(\zeta_i)=c_i e^{\lambda \zeta_i},\qquad i\in\{1,\cdots, n\},$$ for some complex constants $c_i$ and $\lambda$ (independent of $i$).
\end{proposition}
\begin{proof}
Taking the natural logarithm on both sides of \eqref{eq2}, we get
$$\ln f(\zeta_1+\cdots+\zeta_n)=\ln f_1(\zeta_1)+\cdots+\ln f_n(\zeta_n).
$$
Fix any $i\in\{1,\cdots, n\}$ and take partial derivatives with respect to $\zeta_i$ on both sides. The result is
$$\frac{f'(\zeta_1+\cdots+\zeta_n)}{f(\zeta_1+\cdots+\zeta_n)}
=\frac{f'_i(\zeta_i)}{f_i(\zeta_i)}.$$
So for any $i\neq j$ we have
\begin{equation}\label{eq3}
  \frac{f'_i(\zeta_i)}{f_i(\zeta_i)}=\frac{f'_j(\zeta_j)}{f_j(\zeta_j)}.
\end{equation}
Note that the left-hand side of \eqref{eq3} depends only on $\zeta_i$ and the right-hand side depends only on $\zeta_j$,
Therefore, both sides of \eqref{eq3} should be a constant independent of $i$. Then we have
$$f'_i(\zeta_i)=\lambda\ f_i(\zeta_i)$$
for some $\lambda \in\mathbb{C}$, which implies that
$f_i(\zeta_i)=c_i e^{\lambda \zeta_i}$
for some $c_i\in\mathbb{C}$, as desired.
\end{proof}

Throughout the rest of this paper, we are going to use the following notations.
For two multi-indexes $\alpha=(\alpha_1,\cdots,\alpha_n)$ and
$\beta=(\beta_1,\cdots,\beta_n)$ in $\mathbb{N}^n$, we write $\alpha+\beta=(\alpha_1+\beta_1,\cdots,\alpha_n+\beta_n)$ and $\alpha\succeq\beta$ if $\alpha_i\geq\beta_i$ for all $i\in\{1,\cdots,n\}$.
To simplify the notation, we also write $\alpha+1=(\alpha_1+1,\cdots,\alpha_n+1)$.
The reader should have no problem accepting this slightly confusing notation.
In addition, we will need to use the following formulas for integration on $\Omega^n_m$ and $\mathbb{S}_m^n$:
\begin{equation}\label{eq4}
  \int_{\Omega_m^n}f(z)dV(z)=
  \int_0^1r^{2\left(\sum\limits_{i=1}^n{\frac{1}{m_i}}\right)-1}dr
  \int_{\mathbb{S}_m^n}f(r,\zeta)dS(\zeta)
\end{equation}
for every $f\in L^1(\Omega_m^n)$, and
\begin{equation}\label{eq5}
\int_{\mathbb{S}_m^n}\zeta^{\alpha}\overline{\zeta}^{\beta}dS(\zeta)
 =\delta_{\alpha,\beta}
\frac{2\pi^n\prod\limits_{i=1}^n\Gamma\left(\frac{\alpha_i+1}{m_i}\right)}
  {\left(\prod\limits_{i=1}^nm_i\right)\Gamma\left(\sum\limits_{i=1}^n{\frac{\alpha_i+1}{m_i}}\right)},\end{equation}
where $dS$ denotes the hyper surface measure on $\mathbb{S}_m^n$. For more details we refer the reader to \cite{QS}.
Then a direct calculation gives the following simple lemma.

\begin{lemma}\label{and}
Let $l\in\mathbb{R}_+$ and $p,q\in\mathbb{N}^n$. Then on $A^2(\Omega_m^n)$, for each $\beta\in\mathbb{N}^n$, we have
\begin{align*}
T_{r^l\zeta^p\overline{\zeta}^q}(z^{\beta})
&=\left\{
{\begin{array}{*{18}c}
\frac{\Gamma\left(|\widehat{\beta+1}|+|\widehat{p}|-|\widehat{q}|+1\right)
  \prod\limits_{i=1}^n\Gamma\left(\frac{\beta_i+p_i+1}{m_i}\right)}
    {\left(|\widehat{\beta+1}|+\frac{l}{2}+\frac{|\widehat{p}|}{2}-\frac{|\widehat{q}|}{2}\right)
  \Gamma\left(|\widehat{\beta+1}|+|\widehat{p}|\right)
  \prod\limits_{i=1}^n\Gamma\left(\frac{\beta_i+p_i-q_i+1}{m_i}\right)}
\ z^{\beta+p-q},&{\beta+p\succeq q},\\
0,&{\beta+p\not\succeq q}.
\end{array}}
\right.
\end{align*}
\end{lemma}

\begin{proof}
Fix any $\lambda\in\mathbb{N}^n$, then by \eqref{eq4} and \eqref{eq5} we have
\begin{align*}
&\left\langle T_{r^l\zeta^p\overline{\zeta}^q}(z^{\beta}),z^{\lambda}\right\rangle
=\int_{\Omega^n_m}r^l\zeta^p\overline{\zeta}^qz^{\beta}\overline{z}^{\lambda}dV(z)\\
&=\int_{0}^1r^{2\left(\sum\limits_{i=1}^n{\frac{1}{m_i}}\right)+\left(\sum\limits_{i=1}^n{\frac{\beta_i+\lambda_i}{m_i}}\right)+l-1}dr
 \int_{\mathbb{S}^n_m}\zeta^{\beta+p}\overline{\zeta}^{\lambda+q}dS(\zeta)\\
&=\left\{
{
\begin{array}{*{18}c}
\frac{\pi^n\delta_{\beta+p-q,\lambda}
\prod\limits_{i=1}^n\Gamma\left(\frac{\beta_i+p_i+1}{m_i}\right)}
    {\left(\prod\limits_{i=1}^n{m_i}\right)
    \left(\sum\limits_{i=1}^n{\frac{1}{m_i}\left(\beta_i+1+\frac{p_i-q_i}{2}\right)}+\frac{l}{2}\right)
  \Gamma\left(\sum\limits_{i=1}^n{\frac{\beta_i+p_i+1}{m_i}}\right)},&{\beta+p\succeq q},\\
0,&{\beta+p\not\succeq q}.
\end{array}
}
\right.
\end{align*}
Since
$$\left\{\sqrt{{\left(\prod\limits_{i=1}^nm_i\right)
  \Gamma\left(\sum\limits_{i=1}^n{\frac{\alpha_i+1}{m_i}}+1\right)}\bigg/
  {\pi^n\prod\limits_{i=1}^n\Gamma\left(\frac{\alpha_i+1}{m_i}\right)}}\; z^\alpha\right\}_{\alpha\in\mathbb{N}^n}$$
is an orthonormal basis for $A^2(\Omega_m^n)$ (see \cite{CDR} for example), we have if $\beta+p\not\succeq q$ then $T_{r^l\zeta^p\overline{\zeta}^q}(z^{\beta})=0$,
and if $\beta+p\succeq q$ then
\begin{align*}
&T_{r^l\zeta^p\overline{\zeta}^q}(z^{\beta})
=\frac{\left\langle T_{r^l\zeta^p\overline{\zeta}^q}(z^{\beta}),z^{\beta+p-q}\right\rangle}{\left\langle z^{\beta+p-q}, z^{\beta+p-q}\right\rangle}z^{\beta+p-q}\\
& =\frac{\Gamma\left(\sum\limits_{i=1}^n{\frac{\beta_i+1+p_i-q_i}{m_i}}+1\right)
  \prod\limits_{i=1}^n\Gamma\left(\frac{\beta_i+p_i+1}{m_i}\right)}
    {\left(\sum\limits_{i=1}^n{\frac{1}{m_i}\left(\beta_i+1+\frac{p_i-q_i}{2}\right)}+\frac{l}{2}\right)
  \Gamma\left(\sum\limits_{i=1}^n{\frac{\beta_i+p_i+1}{m_i}}\right)
  \prod\limits_{i=1}^n\Gamma\left(\frac{\beta_i+p_i-q_i+1}{m_i}\right)}
\ z^{\beta+p-q}.
\end{align*}
This completes the proof.
\end{proof}

We are now ready to prove Theorem~\ref{thm1} stated in the introduction, which characterizes all
finite rank commutators for monomial-type Toeplitz operators on $A^2(\Omega_m^n)$.
\begin{proof}[Proof of Theorem~\ref{thm1}]
It is trivial that (b) implies (a).

To prove that (a) implies (c), we simply write
\begin{equation}\label{eq6}
    \mu=|\widehat{p}|-|\widehat{q}|,\ \nu=|\widehat{s}|-|\widehat{t}|,\ a=\frac{l}{2}+\frac{|\widehat{p}|}{2}-\frac{|\widehat{q}|}{2},\ b=\frac{k}{2}+\frac{|\widehat{s}|}{2}-\frac{|\widehat{t}|}{2},
\end{equation}
and define
\begin{equation}\label{eq7}
H_{p,q,a}(\xi)=\frac{
  \Gamma\left(\sum\limits_{i=1}^n{\frac{\xi_i+1}{m_i}}+\mu+1\right)
  \prod\limits_{i=1}^n\Gamma\left(\frac{\xi_i+p_i+1}{m_i}\right)}
{\left(\sum\limits_{i=1}^n{\frac{\xi_i+1}{m_i}}+a\right)
  \Gamma\left(\sum\limits_{i=1}^n{\frac{\xi_i+1}{m_i}}+|\widehat{p}|\right)
  \prod\limits_{i=1}^n\Gamma\left(\frac{\xi_i+p_i-q_i+1}{m_i}\right)}.
\end{equation}
Using the same argument as in the proof of \cite[Lemma 3]{DZhu}, we can easily get that $H_{p,q,a}(\xi)$ is holomorphic and polynomially bounded on the domain $$\{\xi\in\mathbb{C}^n:\re \xi_i\geq \max\{0,q_i-p_i\},1\leq i\leq n\}.$$
Then for each $\beta\in\mathbb{N}^n$ with $\beta\succeq \gamma$, where $$\gamma_i=\max\{0,-p_i+q_i,-s_i+t_i,-p_i+q_i-s_i+t_i\}$$
for each $i\in\{1,\cdots,n\}$,
it follows from Lemma~\ref{and} that
\begin{align*}
&\left[T_{r^l\zeta^p\overline{\zeta}^q}, T_{r^k\zeta^s\overline{\zeta}^t}\right](z^\beta)\\
&=\left[ H_{s,t,b}(\beta)H_{p,q,a}(\beta+s-t)-H_{p,q,a}(\beta)H_{s,t,b}(\beta+p-q)\right]\,z^{\beta+p-q+s-t}.
\end{align*}

Assume $\left[T_{r^l\zeta^p\overline{\zeta}^q}, T_{r^k\zeta^s\overline{\zeta}^t}\right]$ has finite rank, so the set
$\left\{\left[T_{r^l\zeta^p\overline{\zeta}^q}, T_{r^k\zeta^s\overline{\zeta}^t}\right](z^\beta):\beta\succeq\gamma\right\}$
contains only finite linearly independent vectors. Thus there exists some $\gamma_0\in\mathbb{N}^n$ such that
$$ H_{s,t,b}(\beta)H_{p,q,a}(\beta+s-t)-H_{p,q,a}(\beta)H_{s,t,b}(\beta+p-q)=0$$ for all $\beta\succeq \gamma_0$.
Clearly, the function
$$ H_{s,t,b}(\xi)H_{p,q,a}(\xi+s-t)-H_{p,q,a}(\xi)H_{s,t,b}(\xi+p-q)$$
is holomorphic and polynomially bounded on $\{\xi\in\mathbb{C}^n: \re \xi_i\geq \gamma_i,1\leq i\leq n\}$. Then according to
\cite[Proposition 3.2]{Le}, we obtain
 \begin{equation}\label{eq8}
   H_{s,t,b}(\xi)H_{p,q,a}(\xi+s-t)-H_{p,q,a}(\xi)H_{s,t,b}(\xi+p-q)=0
\end{equation}
 for all $\xi\in\mathbb{C}^n$ with $\re \xi_i\geq \gamma_i,1\leq i\leq n$.

To simplify our computations later, we define
$$F_i(\eta_i)=
    \frac{
    \Gamma\left(\eta_i+ \frac{p_i}{m_i}\right)
    \Gamma\left( \eta_i+\frac{s_i}{m_i}-\frac{t_i}{m_i}\right)
   \Gamma\left( \eta_i+\frac{p_i}{m_i}-\frac{q_i}{m_i}+\frac{s_i}{m_i}\right)
         }
         {\Gamma\left( \eta_i+\frac{s_i}{m_i}\right)
  \Gamma\left(\eta_i+ \frac{p_i}{m_i}-\frac{q_i}{m_i}\right)
   \Gamma\left( \eta_i+\frac{s_i}{m_i}-\frac{t_i}{m_i}+\frac{p_i}{m_i}\right)}$$
on $\{\eta_i\in\mathbb{C}: \re\eta_i\geq (\gamma_i+1)/m_i\}$ for each $i\in\{1,\cdots,n\}$,
and
$$F(\eta)
 =\frac{\left(\eta+a\right)\left(\eta+b+\mu\right)}
{\left(\eta+b\right)\left(\eta+a+\nu\right)}
\frac{\Gamma\left(\eta+|\widehat{p}|\right)\Gamma\left(\eta+\nu+1\right)\Gamma\left(\eta+\mu+|\widehat{s}|\right)}
{\Gamma\left(\eta+|\widehat{s}|\right)\Gamma\left(\eta+\mu+1\right)\Gamma\left(\eta+\nu+|\widehat{p}|\right)}
$$
on $\{\eta\in\mathbb{C}: \re\eta\geq |\widehat{\gamma+1}|\}$.
Notice that if we replace $\frac{\xi_i+1}{m_i}$ by $\eta_i$ in \eqref{eq7}, then the identity \eqref{eq8}
becomes
 \begin{equation}\label{eq9}
   F(\eta_1+\cdots+\eta_n)=\prod\limits_{i=1}^nF_i(\eta_i), \ \ \re\eta_i\geq(\gamma_i+1)/m_i.
 \end{equation}
Then by Proposition~\ref{uui}, we see that each function $F_i$, $1\leq i\leq n$, should be a constant multiple of the exponential
 function (maybe degenerate to a constant). On the other hand, from the asymptotic expression of the logarithm of gamma function at infinity, it is known that
$$\frac{\Gamma(\eta_i+x)}{\Gamma(\eta_i+y)}=\eta_i^{x-y} \left(1+O\left(\frac{1}{|\eta_i|}\right)\right)$$
for large values of $|\eta_i|$ with $\re\eta_i>0$ and $x,\;y\geq0$.
It is clear that each function $F_i$ is bounded on $\{\eta_i\in\mathbb{C}: \re\eta_i\geq (\gamma_i+1)/m_i\}$. Consequently, we infer that each function $F_i$ is a constant function. Combining this with the definition of $F_i$, we conclude that $F_i(\eta_i)=1$ for each $i\in\{1,\cdots,n\}$.

Observe that $\frac{F_i(\eta_i+1)}{F_i(\eta_i)}=1$. Using the formula $\Gamma\left(\eta_i+1\right)=\eta_i\Gamma\left(\eta_i\right)$, we deduce that
$$\frac{\left(\eta_i+ \frac{p_i}{m_i}\right)
\left( \eta_i+\frac{s_i}{m_i}-\frac{t_i}{m_i}\right)
\left( \eta_i+\frac{p_i}{m_i}-\frac{q_i}{m_i}+\frac{s_i}{m_i}\right)
}{\left( \eta_i+\frac{s_i}{m_i}\right)
\left(\eta_i+ \frac{p_i}{m_i}-\frac{q_i}{m_i}\right)
\left( \eta_i+\frac{s_i}{m_i}-\frac{t_i}{m_i}+\frac{p_i}{m_i}\right)}
=1.$$
This clearly implies that the tuple $(p_i, q_i, s_i, t_i)$ satisfies Condition (I).
Since $F_i(\eta_i)$ is the constant function $1$, it follows from \eqref{eq9} that $F(\eta)=1$, and so
the identity \eqref{eq1} holds on $\{\eta\in\mathbb{C}: \re\eta\geq |\widehat{\gamma+1}|\}$. This
completes the proof that (a) implies (c).

It remains to prove that (c) implies (b). In fact, if condition (c) holds, then $F_i(\eta_i)=1$ and $F(\eta)=1$, and so \eqref{eq9}
holds for $\eta_i$ with $\re\eta_i\geq(\gamma_i+1)/m_i$, $1\leq i\leq n$. Thus,
$$\left[T_{r^l\zeta^p\overline{\zeta}^q}, T_{r^k\zeta^s\overline{\zeta}^t}\right](z^\beta)=0$$
for each $\beta\in\mathbb{N}^n$ with $\beta\succeq\gamma $.
For $\beta\in\mathbb{N}^n$ with $\beta\not\succeq\gamma $, then $\gamma_i>\beta_i\geq0$
for some $i\in\{1,2,\cdots,n\}$. Using \cite[Lemma 1]{DZhu} we have $\gamma_i=-p_i+q_i-s_i+t_i>\beta_i$, and hence Lemma~\ref{and} implies
$$
\left[T_{r^l\zeta^p\overline{\zeta}^q}, T_{r^k\zeta^s\overline{\zeta}^t}\right](z^\beta)=
T_{r^l\zeta^p\overline{\zeta}^q}T_{r^k\zeta^s\overline{\zeta}^t}(z^\beta)
-T_{r^k\zeta^s\overline{\zeta}^t}T_{r^l\zeta^p\overline{\zeta}^q}(z^\beta)=0.
$$
Consequently, $\left[T_{r^l\zeta^p\overline{\zeta}^q}, T_{r^k\zeta^s\overline{\zeta}^t}\right]=0$, as desired.
This completes the proof.
\end{proof}

In the rest of this section, we will prove Theorem~\ref{thm2} stated in the introduction, which completely characterizes all finite rank semi-commutators of two monomial-type Toeplitz operators on $A^2(\Omega^n_m)$. First, we need the following critical lemma.

\begin{lemma}\label{lem5} Let $p,q,s,t\in\mathbb{N}^n$. If there exist some real numbers $a$ and $b$ such that
\begin{equation}\label{eq10}
\frac{\left(\eta+a+b\right)}
{\left(\eta+b\right)
\left(\eta+a+|\widehat{s}|-|\widehat{t}|\right)}
 =\frac{\Gamma\left(\eta+|\widehat{s}|\right)
 \Gamma\left(\eta+|\widehat{s}|-|\widehat{t}|+|\widehat{p}|\right)}
  {\Gamma\left(\eta+|\widehat{s}|-|\widehat{t}|+1\right)
 \Gamma\left(\eta+|\widehat{p}|+|\widehat{s}|\right)}
\end{equation}
hold for any $\eta$ over the right half-plane, then either $|\widehat{t}|=0$ or $|\widehat{p}|=0$.
\end{lemma}

\begin{proof}
Notice that the function on the right-hand side above is a rational function, so it has at most finitely many poles.
Since the function $\Gamma\left(\eta+|\widehat{s}|\right)$ has infinitely many poles, all but finitely many of them must be canceled.
Therefore, either $|\widehat{t}|$ or $|\widehat{p}|$ is a nonnegative integer.

First, we assume that $|\widehat{t}|$ is a positive integer. We need to show that $|\widehat{p}|=0$.  
If $|\widehat{t}|=1$ or $|\widehat{t}|=2$, then an elementary argument shows that 
either $|\widehat{p}|=a=0$ or $|\widehat{p}|=b-|\widehat{s}|+|\widehat{t}|=0$, as desired. So we consider $|\widehat{t}|\geq3$.
Then \eqref{eq10} becomes
\begin{align}\label{eq11}
& \frac{\left(\eta+a+b\right)}
{\left(\eta+b\right)\left(\eta+a+|\widehat{s}|-|\widehat{t}|\right)}\nonumber\\
&=\frac{\left(\eta+|\widehat{s}|-|\widehat{t}|+1\right)\left(\eta+|\widehat{s}|-|\widehat{t}|+2\right)\cdots\left(\eta+|\widehat{s}|-1\right)}
{\left(\eta+|\widehat{p}|+|\widehat{s}|-|\widehat{t}|\right)\cdots
\left(\eta+|\widehat{p}|+|\widehat{s}|-2\right)\left(\eta+|\widehat{p}|+|\widehat{s}|-1\right)}.
\end{align}
Since $|\widehat{t}|\neq0$, it is clear that $|\widehat{p}|\neq1$. To see that $|\widehat{p}|=0$, let us assume that $|\widehat{p}|>0$ and $|\widehat{p}|\neq1$.
Observe that $|\widehat{s}|-|\widehat{t}|+1\neq |\widehat{p}|+|\widehat{s}|-|\widehat{t}|+i$ for any $i\in\{0,\cdots,|\widehat{t}|-1\}$ and $|\widehat{p}|+|\widehat{s}|-1>|\widehat{s}|-1$.
It follows from \eqref{eq11} that
 $$|\widehat{s}|-|\widehat{t}|+1=a+b$$ and
$$|\widehat{p}|+|\widehat{s}|-1=b, \ or\  |\widehat{p}|+|\widehat{s}|-1=a+|\widehat{s}|-|\widehat{t}|.$$
Consequently, one of the following conditions holds:
\begin{itemize}
  \item $a=-|\widehat{p}|-|\widehat{t}|+2$ and $b=|\widehat{p}|+|\widehat{s}|-1$.
  \item $a=|\widehat{p}|+|\widehat{t}|-1$ and $b=-|\widehat{p}|+|\widehat{s}|-2|\widehat{t}|+2$.
\end{itemize}
In each case, \eqref{eq11} becomes
\begin{align*}
&\frac{1}
{\left(\eta-|\widehat{p}|+|\widehat{s}|-2|\widehat{t}|+2\right)}\\
&=\frac{\left(\eta+|\widehat{s}|-|\widehat{t}|+2\right)\cdots\left(\eta+|\widehat{s}|-1\right)}
{\left(\eta+|\widehat{p}|+|\widehat{s}|-|\widehat{t}|\right)
\left(\eta+|\widehat{p}|+|\widehat{s}|-|\widehat{t}|+1\right)\cdots\left(\eta+|\widehat{p}|+|\widehat{s}|-2\right)}.
\end{align*}
However, $$-|\widehat{p}|+|\widehat{s}|-2|\widehat{t}|+2\neq|\widehat{p}|+|\widehat{s}|-|\widehat{t}|+j$$ for all $j\in\{0,\cdots,|\widehat{t}|-2\}$ since $|\widehat{t}|\geq3$ and $|\widehat{p}|>0$, a contradiction. This shows that $|\widehat{p}|=0$.

Next, we suppose $|\widehat{p}|$ is a positive integer.
Then by the same way as shown before, we get that $|\widehat{t}|=0$.
This completes the proof.
\end{proof}

\begin{proof}[Proof of Theorem~\ref{thm2}]
It is trivial that (b) implies (a).

To prove that (a) implies (c), we consider each $\beta\in\mathbb{N}^n$
with $\beta\succeq\delta$, where $$\delta_i=\max\{0,-s_i+t_i,-p_i+q_i-s_i+t_i\}.$$ Then
we deduce from Lemma~\ref{and} and the notation of \eqref{eq7} that
\begin{align*}
&\left(T_{r^l\zeta^p\overline{\zeta}^q}, T_{r^k\zeta^s\overline{\zeta}^t}\right](z^\beta)\\
&=\left[ H_{s,t,b}(\beta)H_{p,q,a}(\beta+s-t)-H_{p+s,q+t,a+b}(\beta)\right](z^{\beta+p-q+s-t}).
\end{align*}
Since $\left(T_{r^l\zeta^p\overline{\zeta}^q}, T_{r^k\zeta^s\overline{\zeta}^t}\right]$ has finite rank, we can
proceed as in the proof of Theorem~\ref{thm1} to obtain
\begin{equation}\label{eq12}
  H_{s,t,b}(\xi)H_{p,q,a}(\xi+s-t)-H_{p+s,q+t,a+b}(\xi)=0
\end{equation}
for all $\xi\in\mathbb{C}^n$ with $\re\xi_i\geq\delta_i, 1\leq i\leq n$.
If we define
$$G_i(\eta_i)=
  \frac{\Gamma\left( \eta_i+\frac{s_i}{m_i}-\frac{t_i}{m_i}\right)
    \Gamma\left( \eta_i+\frac{p_i}{m_i}+\frac{s_i}{m_i}\right)}
  {\Gamma\left( \eta_i+\frac{s_i}{m_i}\right)
   \Gamma\left( \eta_i+\frac{s_i}{m_i}-\frac{t_i}{m_i}+\frac{p_i}{m_i}\right)}
$$
on $\{\eta_i\in\mathbb{C}: \re\eta_i\geq (\delta_i+1)/m_i\}$ for each $i\in\{1,\cdots,n\}$ and
$$G(\eta)=
\frac{\left(\eta+a+b\right)\Gamma\left(\eta+\nu+1\right)\Gamma\left(\eta+|\widehat{p}|+|\widehat{s}|\right)}
{\left(\eta+b\right)\left(\eta+a+\nu\right)\Gamma\left(\eta+|\widehat{s}|\right)\Gamma\left(\eta+\nu+|\widehat{p}|\right)}
$$
on $\{\eta\in\mathbb{C}: \re\eta\geq |\widehat{\delta+1}|\}$, then \eqref{eq12} is equivalent to
\begin{equation*}
   G(\eta_1+\cdots+\eta_n)=\prod\limits_{i=1}^n G(\eta_i)
\end{equation*}
for all $\eta_i\in\mathbb{C}$ with $\re \eta_i\geq (\delta_i+1)/m_i,1\leq i\leq n$.
Using the same argument as in the proof of Theorem~\ref{thm1}, we have that $G(\eta)=1$. Then by Lemma~\ref{lem5}, we get either $|\widehat{t}|=0$ or $|\widehat{p}|=0$.
If $|\widehat{t}|=0$, then it follows from $G(\eta)=1$ that
$$\left(\eta+a+b\right)\left(\eta+|\widehat{s}|\right)
=\left(\eta+b\right)\left(\eta+a+|\widehat{s}|\right),$$ which implies that either $a=0$ or $b=|\widehat{s}|$. Similarly, if $|\widehat{p}|=0$, then we have
either $a=0$ or $b=\nu$. This
completes the proof that (a) implies (c).

To prove that (c) implies (b), we first observe that condition (c) implies that either $p=0$ or $t=0$, and consequently,
$G_i(\eta_i)=1$ on $\{\eta_i\in\mathbb{C}: \re\eta_i\geq (\delta_i+1)/m_i\}$ for each $i\in\{1,\cdots,n\}$. It is also easy to check that $G(\eta)=1$, and so \eqref{eq12}
holds for for all $\zeta\in\mathbb{C}^n$ with $\re\zeta_i\geq\delta_i, 1\leq i\leq n$. Thus,
$$\left(T_{r^l\zeta^p\overline{\zeta}^q}, T_{r^k\zeta^s\overline{\zeta}^t}\right](z^\beta)=0$$
for each $\beta\in\mathbb{N}^n$ with $\beta\succeq\delta$.
For $\beta\in\mathbb{N}^n$ with $\beta\not\succeq\delta$, then $\delta_{i_0}>\beta_{i_0}\geq0$
for some $i_0\in\{1,2,\cdots,n\}$. As in the proof of Theorem~\ref{thm1}, condition (b) will follow if we can show that
$$\delta_{i_0}=-p_{i_0}+q_{i_0}-s_{i_0}+t_{i_0}.$$
To this end, first observe that either $p_{i_0}=0$ or $t_{i_0}=0$. If $p_{i_0}=0$, then it is obvious that $$-s_{i_0}+t_{i_0}\leq-p_{i_0}+q_{i_0}-s_{i_0}+t_{i_0}.$$
Since $\delta_{i_0}>0$, the desired result then follows from the definition of $\delta_{i_0}$. If $t_{i_0}=0$, then $\delta_{i_0}=\max\{0,-s_{i_0},-p_{i_0}+q_{i_0}-s_{i_0}\}$.
Since $\delta_{i_0}>0$, the desired result is then obvious.
This completes the proof.
\end{proof}

\section{Corollaries and examples}

In this section we observe some interesting consequences of Theorem~\ref{thm1}. More specifically,
we will give some sufficient conditions and necessary conditions for two monomial-type Toeplitz operators $T_{r^l\zeta^p\overline{\zeta}^q}$ and $T_{r^k\zeta^s\overline{\zeta}^t}$ commuting on $A^2(\Omega_m^n)$. For the convenience of our proofs, we rewrite \eqref{eq1} as
\begin{align}\label{eq13}
&\frac{\Gamma\left(\eta+|\widehat{p}|\right)
\Gamma\left(\eta+|\widehat{s}|-|\widehat{t}|+1\right)
\Gamma\left(\eta+|\widehat{p}|-|\widehat{q}|+|\widehat{s}|\right)}{\Gamma\left(\eta+|\widehat{s}|\right)
 \Gamma\left(\eta+|\widehat{p}|-|\widehat{q}|+1\right)
\Gamma\left(\eta+|\widehat{s}|-|\widehat{t}|+|\widehat{p}|\right)}\nonumber\\
 &=\frac{\left(\eta+\frac{k}{2}+\frac{|\widehat{s}|}{2}-\frac{|\widehat{t}|}{2}\right)
\left(\eta+\frac{l}{2}+|\widehat{s}|-|\widehat{t}|+\frac{|\widehat{p}|}{2}
-\frac{|\widehat{q}|}{2}\right)}{\left(\eta+\frac{l}{2}+\frac{|\widehat{p}|}{2}-\frac{|\widehat{q}|}{2}\right)
\left(\eta+\frac{k}{2}+|\widehat{p}|-|\widehat{q}|+\frac{|\widehat{s}|}{2}
-\frac{|\widehat{t}|}{2}\right)}.
\end{align}

As a direct consequence of Theorem~\ref{thm1}, we first present five cases for two monomial-type
Toeplitz operators commuting on $A^2(\Omega_m^n)$, which correspond to \cite[Theorem A]{DZhu}.

\begin{corollary}\label{cor five} If one of the following conditions holds, then the operators $T_{r^l\zeta^p\overline{\zeta}^q}$ and $T_{r^k\zeta^s\overline{\zeta}^t}$ commute on $A^2(\Omega_m^n)$:
\begin{itemize}
\item[(c1)] $l=|\widehat{p}|=|\widehat{q}|=0$ or $k=|\widehat{s}|=|\widehat{t}|=0$,
            either of which means that one of the two operators is the identity operator.
\item[(c2)] $|\widehat{q}|=|\widehat{t}|=0$, $l=|\widehat{p}|$ and $k=|\widehat{t}|$,
            which means that both operators have analytic symbols.
\item[(c3)] $|\widehat{p}|=|\widehat{s}|=0$, $l=|\widehat{q}|$ and $k=|\widehat{s}|$, which means that both operators have
             conjugate analytic symbols.
\item[(c4)] $\left(p_i, q_i, s_i, t_i\right)$ satisfies Condition (I) for each $i\in\{1,2,\cdots,n\}$,
           $|\widehat{p}|=|\widehat{q}|$ and $|\widehat{s}|=|\widehat{t}|$.
\item[(c5)] $\left(p_i, q_i, s_i, t_i\right)$ satisfies Condition (I) for each $i\in\{1,2,\cdots,n\}$,
           $|\widehat{p}|=|\widehat{s}|$ , $|\widehat{q}|=|\widehat{t}|$, and $l=k$.
\end{itemize}
\end{corollary}

It is also worth to mention that except the above predictable sufficient conditions for two monomial-type
Toeplitz operators commuting on $A^2(\Omega_m^n)$, there are exactly many other cases. Next, we present some non-trivial examples of commuting monomial-type
Toeplitz operators on $A^2(\Omega_m^n)$.

\begin{example}\label{ex nontri}
Let $\left(p_i, q_i, s_i, t_i\right)$ satisfy Condition (I) for each $i\in\{1,2,\cdots,n\}$, and let
\begin{align*}
\left\{
{
\begin{array}{*{18}c}
|\widehat{p}|=2,\\|\widehat{q}|=1,\\|\widehat{s}|=2|\widehat{t}|,\\
l=2|\widehat{t}|-1,\\
k=|\widehat{t}|,\\
\end{array}
}
\right.
or \ \left\{
{
\begin{array}{*{18}c}
|\widehat{p}|=2,\\|\widehat{q}|=1,\\|\widehat{s}|=2|\widehat{t}|,\\
l=2|\widehat{t}|+1,\\
k=3|\widehat{t}|,
\end{array}
}
\right.
or \ \left\{
{
\begin{array}{*{18}c}
|\widehat{q}|=2,\\|\widehat{t}|=1,\\|\widehat{p}|=2|\widehat{s}|,\\
l=4|\widehat{s}|-2,\\
k=3|\widehat{s}|-1.
\end{array}
}
\right.
\end{align*}
Then it is easy to check that $T_{r^l\zeta^p\overline{\zeta}^q}$ commute with $T_{r^k\zeta^s\overline{\zeta}^t}$
on $A^2(\Omega_m^n)$ in each of the cases above. More specifically, if we define $p, q, s, t\in\N^3$ by
\begin{align*}
    p=(m_1,m_2,0),\ q=(m_1,0,0),\
    s=(2m_1,2m_2,4m_3),\
    t=(2m_1,0,2m_3),
\end{align*}
then $(|\widehat{p}|,|\widehat{q}|,|\widehat{s}|,|\widehat{t}|)=(2,1,8,4)$, which does not satisfy Corollary~\ref{cor five}.
Obviously, $T_{r^7\zeta^p\overline{\zeta}^q}$ commutes with $T_{r^{4}\zeta^s\overline{\zeta}^t}$, and $T_{r^9\zeta^p\overline{\zeta}^q}$ commutes with $T_{r^{12}\zeta^s\overline{\zeta}^t}$ on $A^2(\Omega_m^3)$.
\end{example}

If the weakly pseudoconvex domain is not the unit ball, we can construct the following more interesting example.

\begin{example}\label{ex nontri2} Suppose $m=(4,\cdots,4)\in\mathbb{N}^6$ and
\begin{align*}
    p&=(0,2,0,1,1,4),\\
    q&=(0,1,1,0,1,1),\\
    s&=(2,0,0,8,2,4),\\
    t&=(3,0,2,0,2,1).
\end{align*}
Then $(|\widehat{p}|,|\widehat{q}|,|\widehat{s}|,|\widehat{t}|)=(2,1,4,2)$ and the tuple $\left(p_{i}, q_{i}, s_{i}, t_{i}\right)$ satisfies Condition (I) for all $i\in\{1,\cdots,6\}$. As a consequence
of Example~\ref{ex nontri}, we see that the operators $T_{z^p\overline{z}^q}$ and $T_{\zeta^s\overline{z}^t}$ commute
on $A^2(\Omega_m^6)$.
\end{example}

Next, we give some interesting necessary conditions for two monomial-type Toeplitz operators commuting on $A^2(\Omega_m^n)$.

\begin{corollary}\label{cor nec}
Let $l,k\in\mathbb{R}_+$, $p,q,s,t\in\mathbb{N}^n$. If the operators $T_{r^l\zeta^p\overline{\zeta}^q}$ and $T_{r^k\zeta^s\overline{\zeta}^t}$ commute on $A^2(\Omega_m^n)$, then the following statements hold:
\begin{enumerate}
  \item[(a)] At least one of the tuple $\left(|\widehat{p}|,|\widehat{q}|\right)$, $\left(|\widehat{s}|,|\widehat{t}|\right)$, $\left(|\widehat{p}|,|\widehat{s}|\right)$, $\left(|\widehat{q}|,|\widehat{t}|\right)$, $\left(|\widehat{p}|-|\widehat{q}|, |\widehat{s}|-|\widehat{t}|\right)$, and $\left(|\widehat{p}|-|\widehat{s}|, |\widehat{q}|-|\widehat{t}|\right)$ belongs to $\mathbb{Z}^2$.
  \item[(b)] The tuple $(l,k,|\widehat{p}|,|\widehat{q}|,|\widehat{s}|,|\widehat{t}|)$ must satisfy the following equation:
\begin{align}\label{eq14}
&\frac{\left(\eta+|\widehat{p}|\right)
\left(\eta+|\widehat{s}|-|\widehat{t}|+1\right)
\left(\eta+|\widehat{p}|-|\widehat{q}|+|\widehat{s}|\right)}
{\left(\eta+|\widehat{s}|\right)\left(\eta+|\widehat{p}|-|\widehat{q}|+1\right)
\left(\eta+|\widehat{s}|-|\widehat{t}|+|\widehat{p}|\right)}\nonumber\\
&=\frac{\left(\eta+b+1\right)\left(\eta+a+|\widehat{s}|-|\widehat{t}|+1\right)\left(\eta+a\right)\left(\eta+b+|\widehat{p}|-|\widehat{q}|\right)}
{\left(\eta+b\right)\left(\eta+a+|\widehat{s}|-|\widehat{t}|\right)\left(\eta+a+1\right)\left(\eta+b+|\widehat{p}|-|\widehat{q}|+1\right)},
 \end{align}
where $a=\left(l+|\widehat{p}|-|\widehat{q}|\right)/2$ and $ b=\left(k+|\widehat{s}|-|\widehat{t}|\right)/2$.
\end{enumerate}
\end{corollary}
\begin{proof} To prove part (a), let us assume the contrary. Then the analytic function on the right-hand side of \eqref{eq13} has no zero on the complex plane.
However, if the left-hand side of \eqref{eq13} has no zero, then it follows
\begin{align*}
&\left(\eta+\frac{k}{2}+\frac{|\widehat{s}|}{2}-\frac{|\widehat{t}|}{2}\right)
\left(\eta+\frac{l}{2}+|\widehat{s}|-|\widehat{t}|+\frac{|\widehat{p}|}{2}
-\frac{|\widehat{q}|}{2}\right)\\
&=\left(\eta+\frac{l}{2}+\frac{|\widehat{p}|}{2}-\frac{|\widehat{q}|}{2}\right)
\left(\eta+\frac{k}{2}+|\widehat{p}|-|\widehat{q}|+\frac{|\widehat{s}|}{2}
-\frac{|\widehat{t}|}{2}\right).
\end{align*}
From the assumption $\left(|\widehat{p}|-|\widehat{q}|, |\widehat{s}|-|\widehat{t}|\right)\notin\mathbb{Z}^2$ we then deduce that $k=l$ and
\begin{equation}\label{eq15}
    |\widehat{p}|-|\widehat{q}|=|\widehat{s}|-|\widehat{t}|
\end{equation}
Then the identity \eqref{eq13} becomes
$$\Gamma\left(\eta+|\widehat{p}|\right)
\Gamma\left(\eta+|\widehat{p}|-|\widehat{q}|+|\widehat{s}|\right)
=\Gamma\left(\eta+|\widehat{s}|\right)
\Gamma\left(\eta+|\widehat{s}|-|\widehat{t}|+|\widehat{p}|\right).$$
This together with \eqref{eq15} shows that either $|\widehat{p}|-|\widehat{q}|=|\widehat{s}|-|\widehat{t}|=0$ or
           $|\widehat{p}|-|\widehat{s}|=|\widehat{q}|-|\widehat{t}|=0$, which contradicts with the assumption and
completes the proof of the condition (a).

To prove part (b), we can replace $\eta$ by $\eta+1$ in \eqref{eq13} and apply the formula $\Gamma\left(\eta+1\right)=\eta\Gamma\left(\eta\right)$ to obtain
\begin{align*}
&\frac{\Gamma\left(\eta+|\widehat{p}|\right)
\Gamma\left(\eta+|\widehat{s}|-|\widehat{t}|+1\right)
\Gamma\left(\eta+|\widehat{p}|-|\widehat{q}|+|\widehat{s}|\right)}
{\Gamma\left(\eta+|\widehat{s}|\right)\Gamma\left(\eta+|\widehat{p}|-|\widehat{q}|+1\right)
\Gamma\left(\eta+|\widehat{s}|-|\widehat{t}|+|\widehat{p}|\right)}\\
&=\frac
{\left(\eta+b+1\right)\left(\eta+a+|\widehat{s}|-|\widehat{t}|+1\right)\left(\eta+|\widehat{s}|\right)
\left(\eta+|\widehat{p}|-|\widehat{q}|+1\right)\left(\eta+|\widehat{s}|-|\widehat{t}|+|\widehat{p}|\right)}
{\left(\eta+a+1\right)\left(\eta+b+|\widehat{p}|-|\widehat{q}|+1\right)\left(\eta+|\widehat{p}|\right)
\left(\eta+|\widehat{s}|-|\widehat{t}|+1\right)\left(\eta+|\widehat{p}|-|\widehat{q}|+|\widehat{s}|\right)}.
 \end{align*}
The desired result then follows from \eqref{eq13} and the above identity.
\end{proof}

Of course, one may expect that there should exist some differences in
operator theory on the Bergman spaces between on the weakly pseudoconvex domain $\Omega^n_m$ and on the unit ball $\mathbb{B}^n$.
Note that $|\widehat{p}|$, $|\widehat{q}|$, $|\widehat{s}|$ and $|\widehat{t}|$ will always satisfy condition (a) of Corollary~\ref{cor nec} for the case $m=(1,\cdots,1)$. In many cases, however, this is no longer true. 

In addition, it is interesting to observe that the collection of the tuple $(l,k,|\widehat{p}|,|\widehat{q}|,|\widehat{s}|,|\widehat{t}|)$ satisfying
\eqref{eq14} is finite. Consequently, there exist a finite number of the combination of the tuple $(l,k,|\widehat{p}|,|\widehat{q}|,|\widehat{s}|,|\widehat{t}|)$ such that the operators $T_{r^l\zeta^p\overline{\zeta}^q}$ and $T_{r^k\zeta^s\overline{\zeta}^t}$ commute on $A^2(\Omega_m^n)$. Moreover, with the help of \eqref{eq14}, we can easily obtain the next two corollaries, which give the specific sufficient and necessary condition for some special monomial-type Toeplitz operators to be commutitive on $A^2(\Omega_m^n)$.
The first corollary characterizes commuting monomial Toeplitz operators on the Bergman space of the weakly pseudoconvex domains, which is the same as the case on the unit ball.
\begin{corollary}\label{five}
Let $p,q,s,t\in\mathbb{N}^n$. Then the operators $T_{z^p\overline{z}^q}$ and $T_{z^s\overline{z}^t}$ commute on $A^2(\Omega_m^n)$ if and only if $(|\widehat{p}|,|\widehat{q}|,|\widehat{s}|,|\widehat{t}|)$ and $(p_i,q_i,s_i,t_i)$
       satisfy Condition (I) for all $i\in\{1,2,\cdots,n\}$.
\end{corollary}
\begin{proof}
Denote $l=|\widehat{p}|+|\widehat{q}|$ and $k=|\widehat{s}|+|\widehat{t}|$. Then the identity \eqref{eq14} becomes
\begin{align*}
& \left(\eta+|\widehat{s}|-|\widehat{t}|+1\right)\left(\eta+|\widehat{p}|+1\right)\left(\eta+|\widehat{s}|+|\widehat{p}|-|\widehat{q}|+1\right)\\
&=\left(\eta+|\widehat{p}|-|\widehat{q}|+1\right)\left(\eta+|\widehat{s}|+1\right)
\left(\eta+|\widehat{p}|+|\widehat{s}|-|\widehat{t}|+1\right),
\end{align*}
which implies that $(|\widehat{p}|,|\widehat{q}|,|\widehat{s}|,|\widehat{t}|)$ satisfies Condition (I).
The desired result then follows from Theorem~\ref{thm1}.
\end{proof}
Our next corollary shows that a Toeplitz operator
with a holomorphic monomial symbol may only commute with another monomial-type Toeplitz operator with a holomorphic
symbol.
\begin{corollary}
Let $k\in\mathbb{R}_+$, $p,s,t\in\mathbb{N}^n$ with $p\neq0$. Then the operators $T_{z^p}$ and $T_{r^k\zeta^s\overline{\zeta}^t}$ commute on $A^2(\Omega_m^n)$ if and only if
the operator $T_{r^k\zeta^s\overline{\zeta}^t}$ also has holomorphic symbol.
\end{corollary}
\begin{proof}
First assume that $T_{z^p}$ and $T_{r^k\zeta^s\overline{\zeta}^t}$ commute on $A^2(\Omega_m^n)$. Then by the identity \eqref{eq14} we have
\begin{align*}
&\frac{\left(\eta+|\widehat{s}|-|\widehat{t}|+1\right)\left(\eta+\frac{k}{2}+\frac{|\widehat{s}|}{2}-\frac{|\widehat{t}|}{2}\right)
}
{\left(\eta+|\widehat{s}|\right)\left(\eta+\frac{k}{2}+\frac{|\widehat{s}|}{2}-\frac{|\widehat{t}|}{2}+1\right)}\\
&=\frac{\left(\eta+|\widehat{p}|+|\widehat{s}|-|\widehat{t}|+1\right)\left(\eta+\frac{k}{2}+\frac{|\widehat{s}|}{2}-\frac{|\widehat{t}|}{2}+|\widehat{p}|\right)}
{\left(\eta+|\widehat{p}|+|\widehat{s}|\right)
\left(\eta+\frac{k}{2}+\frac{|\widehat{s}|}{2}-\frac{|\widehat{t}|}{2}+|\widehat{p}|+1\right)}.
\end{align*}
Thus the function on the left-hand side of the above identity is a bounded periodic analytic function with a period $|\widehat{p}|$ in the right half-plane,
and consequently, it must be identity function. So we have
$$\left(\eta+|\widehat{s}|-|\widehat{t}|+1\right)\left(\eta+\frac{k}{2}+\frac{|\widehat{s}|}{2}-\frac{|\widehat{t}|}{2}\right)
=\left(\eta+|\widehat{s}|\right)\left(\eta+\frac{k}{2}+\frac{|\widehat{s}|}{2}-\frac{|\widehat{t}|}{2}+1\right),$$
which implies that $|\widehat{t}|=0$ and $k=|\widehat{s}|$, as desired.

By (c2) of Corollary~\ref{cor five}, the converse implication is clear. This completes the proof.
\end{proof}


\begin{thebibliography}{1}
\bibitem{ARS} M. Abate, J. Raissy and A. Saracco,  \textit{Toeplitz operators and Carleson measures in strongly pseudoconvex domains}, J. Funct. Anal. \textbf{263} (2012), 3449--3491.
\bibitem{AhC} P. Ahern and \v Z. \v Cu\v ckovi\' c, \emph{A theorem of Brown-Halmos type for Bergman space Toeplitz operators}, J. Funct. Anal. \textbf{187} (2001), 200--210.
\bibitem{AxC} S. Axler and \v Z. \v Cu\v ckovi\' c,
\textit{Commuting Toeplitz operators with harmonic symbols},
Integr. Equat. Oper. Theory. \textbf{14} (1991), 1--12.
\bibitem{CR} \v Z. \v Cu\v ckovi\' c and N.V. Rao, \textit{Mellin transform, monomial symbols, and
commuting Toeplitz operators}, J. Funct. Anal. \textbf{154} (1998),
195--214.
\bibitem{CDR} D. Crocker and I. Raeburn, \textit{Toeplitz operators on certain weakly pseudoconvex domains},
             J. Aust. Math. Soc. Ser. A. \textbf{31} (1981),1-14.
\bibitem{DZ1} X.T. Dong and Z.H. Zhou, Algebraic properties of Toeplitz operators with separately
quasihomogeneous symbols on the Bergman space of the unit ball, \textit{J. Operator Theory}
\textbf{66} (2011), 193--207.
\bibitem{DZ6} X.T. Dong and Z.H. Zhou, \textit{Ranks of commutators and generalized semi-commutators of
        quasi-homogeneous Toeplitz operators}, Monatsh. Math. \textbf{183} (2017), 103-141.
\bibitem{DZhu} X.T. Dong and K. Zhu, \textit{Commutators and semi-commutators of Toeplitz
          operators on the unit ball}, Integr. Equat. Oper. Theory. \textbf{86} (2016), 271-300.
\bibitem{GSZ} K. Guo, S. Sun and D. Zheng, \textit{Finite rank commutators and semicommutators
            of Toeplitz operators with harmonic symbols}. Ill. J. Math. \textbf{51} (2007), 583--596.
\bibitem{JK} N. P. Jewell and S. G. Krantz, \textit{Toeplitz operators and related function algebras on certain
pseudoconvex domains}, Trans. Amer. Math. Soc. \textbf{252} (1979),297--312.
\bibitem{Le} T. Le, \textit{The commutants of certain Toeplitz operators on weighted Bergman spaces},
              J. Math. Anal. Appl. \textbf{348} (2008), 1--11.
\bibitem{Le1} T. Le, \textit{Commutants of separately radial Toeplitz operators in several variables}, J. Math. Anal. Appl. \textbf{453} (2017), 48--63.
\bibitem{QS} R. Quiroga-Barranco and A. Sanchez-Nungaray,
        \textit{Toeplitz operators with quasi-Homogeneous quasi-radial symbols on some weakly pseudoconvex domains},
        Complex Anal. Oper. Theory. \textbf{9} (2015), 1111--1134.
\bibitem{LSZ} I. Louhichi, E. Strouse and L. Zakariasy, \textit{Products of Toeplitz operators on the Bergman space}, Integr. Equat. Oper. Theory. \textbf{54} (2006), 525--539.
\bibitem{LZ1} I. Louhichi and L. Zakariasy, \textit{On Toeplitz operators with quasihomogeneous symbols}, Arch. Math. \textbf{85} (2005), 248--257.
\bibitem{V} N. Vasilevski, \textit{Quasi-radial quasi-homogeneous symbols and commutative Banach algebras of
Toeplitz operators}, Integr. Equat. Oper. Theory. \textbf{66} (2010), 141--152.
\bibitem{W} S.M. Webster, \textit{Biholomorphic mappings and the Bergmann kernel off the diagonal}, Invent.
Math. \textbf{51} (1979), 155--169.
\bibitem{Z1} D. Zheng, \textit{Commuting Toeplitz Operators with Pluriharmonic Symbols}, Trans. Amer. Math.
Soc. \textbf{350} (1998), 1595--1618.
\bibitem{ZD} Z.H. Zhou and X.T. Dong, \textit{Algebraic properties of Toeplitz
operators with radial symbols on the Bergman space of the unit
ball}, Integr. Equat. Oper. Theory. \textbf{64} (2009),
137--154.
\end{thebibliography}
  \end{document}